\documentclass[14pt,a4paper]{article}
\usepackage{ifthen}
\usepackage{tikz}
\usepackage{amsmath}
\usepackage{amscd}
\usepackage{amssymb}
\usepackage{lipsum}
\usepackage{mathrsfs}
\usepackage{tikz-cd}
\usepackage{tikz}
\usepackage{makeidx}
\usepackage{latexsym}
\usepackage[parfill]{parskip}
\usepackage{graphicx}
\usepackage[english]{babel}
\usepackage{amsthm}

\usepackage{setspace}
\theoremstyle{plain}
\newtheorem{thm}{Theorem}[section]

\newtheorem{prop}[thm]{Proposition}

\theoremstyle{remark}

\newcommand{\thmref}[1]{Theorem~\ref{#1}}

\newcommand{\R}{\mathbb{R}}
\newcommand{\C}{\mathbb{C}}

\renewcommand{\=}{\ \approx \ }

\renewcommand{\log}{\textrm{log}\,}

\newcommand{\pdo}{\psi{\rm do}}

\newcommand{\si}{^{-s}}

\newcommand{\noi}{\noindent}

\newcommand{\ol}{\overline}

\newcommand{\wh}{\widehat}

\newcommand{\x}{\times}
\newcommand{\ox}{\otimes}

\newcommand{\aand}{\mbox{and}}

\newcommand{\ffor}{\mbox{for}}

\newcommand{\bsh}{\backslash}

\newcommand{\Cc}{{\mathcal C}}

\newcommand{\Mm}{{\mathcal M}}

\newcommand{\rs}{{\textsf{r}}}

\newcommand{\vtwo}{\vskip 2mm}

\newcommand{\vfive}{\vskip 5mm}

\renewcommand{\r}{\right}
\renewcommand{\l}{\left}

\newcommand{\ii}{^{-1}}
\newcommand{\pr}{^{\prime}}

\renewcommand{\d}{\delta}
\newcommand{\D}{\Delta}

\renewcommand{\log}{\mbox{{\rm log}\,}}

\newcommand{\z}{\zeta}
\newcommand{\s}{\sigma}

\renewcommand{\b}{\beta}

\newcommand{\g}{\gamma}

\newcommand{\la}{\lambda}

\renewcommand{\d}{\delta}

\newcommand{\Z}{\mathbb{Z}}

\renewcommand{\O}{\Omega}

\newcommand{\oo}{\infty}

\newcommand{\End}{\mbox{\rm End}\,}

\newcommand{\rk}{\mbox{\rm rk}}

\newcommand{\tr}{\mbox{\rm tr\,}}

\newcommand{\Tr}{\mbox{\rm Tr\,}}

\newcommand{\TR}{\mbox{\textsf{\small  TR}}}

\renewcommand{\Re}{{\rm Re}}

\newcommand{\res}{\textrm{res}}

\numberwithin{equation}{section}

\begin{document} 

\title{ \textsc{Residue-Torsion and the Laplacian on Riemannian Manifolds}}

\author{Niccol\`{o} Salvatori and Simon Scott}
\date{}

\maketitle

{\large The Ray-Singer \cite{RaSi}  character formula
\begin{equation}\label{R torsion}
\log R_{X}(\rho)    = \frac{1}{2}\sum _{k= 0}^m (-1)^{k+1} k\, \tr(\log \Delta^c_k) 
\end{equation}
for computing the Reidemeister torsion  $R_X(\rho)$ of a closed oriented acyclic Riemannian manifold  $X$ with flat vector bundle $E_\rho\to X $ poses the question of whether  \eqref{R torsion}  may define analogous generalised torsion invariants for other elliptic complexes. Here, $E_\rho$ is specified by  a representation $\rho:  \pi_1(X ) \to O(n)$ and 
 $\Delta^c_k := \d_{k-1}\d_{k-1}^* + \d_k^* \d_k$ 
 is the combinatorial Laplacian of the acyclic chain complex $\d_k : C_k (\wh{X })\ox_{\,\R\pi_1(X )} \R^n \to C_{k-1} (\wh{X })\ox_{\,\R\pi_1(X )} \R^n$ over the universal cover $\wh{X }$, with logarithm operator
 \begin{equation}\label{cx power log}
 \log \Delta^c_k  := \frac{d}{dz} (\Delta^c_k)^z\left.\right|_{z=0}
 \end{equation}
where the complex power  $(\Delta^c_k)^z$ of the (strictly positive invertible finite-rank real) matrix $\Delta^c_k$ is defined canonically via holomorphic functional calculus.

Replicating \eqref{R torsion} for the de Rham complex $ \O^k(X ,E_\rho)  \stackrel{d_k}{\to} \O^{k+1}(X ,E_\rho)$
 with coefficients in $E_\rho$, Ray and Singer  \cite{RaSi} defined the generalised torsion 
\begin{equation}\label{analytic torsion}
\log T_{X}^{\,\z, {\bf k}}(\rho) := \frac{1}{2}\sum _{k= 0}^m (-1)^{k+1} k\, \TR_\z(\log \Delta_k)
\end{equation}
for the Hodge Laplacian $\Delta_k := d_{k-1} d^*_{k-1}  + d^*_k d_k :  \O^k(X ,E_\rho)\to  \O^k(X ,E_\rho)$ with  logarithm  $\log \Delta_k$ defined, similarly to \eqref{cx power log}, as the derivative of the complex power $\Delta_k^z$ and with  $\TR_\z$ the quasi-trace extension of the classical trace obtained by zeta function regularization.  The analytic torsion $\log T_{X}^{\,\z, {\bf k}}(\rho)$ was shown to be trivial for even-dimensional $X $  and in odd dimensions to be  a smooth invariant if the cohomology with coefficients in $E_\rho$ is trivial (acyclic). The general equality  ${T_{X}^{\,\z, {\bf k}}(\rho)  =R_X(\rho)}$ conjectured by Ray and Singer  \cite{RaSi}  was subsequently proved by Cheeger \cite{Ch} and by Muller \cite{Mu}.
%, showing analytic torsion to define a topological invariant. 

In this paper we examine the generalised torsion 
\begin{equation}\label{residue torsion}
 \log T_{X}^{\,\res,  {\bf k}}(\rho)  := \frac{1}{2}\sum _{k= 0}^m (-1)^{k+1} k\, \res(\log \Delta_k)
\end{equation} 
of the de Rham complex  defined using the {\it residue trace} `$\res$' in place of the zeta function trace $\TR_\z$.  Residue torsion \eqref{residue torsion} is roughly complementary to analytic torsion \eqref{analytic torsion} --- it is trivial in odd dimensions while in even dimensions it is a non-trivial topological (in fact, homotopy) invariant  which can  be non-zero only when the cohomology is non-trivial (non acyclic). This is consistent with  the residue trace being roughly complementary to the classical (zeta) trace.  Residue trace invariants are relatively elementary, depending locally on only finitely many terms in the symbol expansion of a pseudodifferential operator, whilst the classical trace, and hence the zeta trace, is globally determined. The residue torsion is hence a far simpler invariant than the subtle and difficult to compute analytic (Reidemeister) torsion, and this is seen in the following exact identifications.

 \begin{thm}\label{thm one}
Let $X$ be a closed manifold of dimension $n$. Let $\b=(\b_0,\dots, \b_n)\in\R^{n+1}$. The generalised torsion 
$\log T_X^{\,\res, {\bf \b}}(\rho) := \frac{1}{2}\sum _{k= 0}^n (-1)^{k+1} \b_k\, 
{\rm res}(\log \Delta_k)$
is a smooth and topological invariant if and only if, up to a constant multiple, $\b_k=1$ for each $k$ or if $\b_k=k$ for each $k$. 
Writing $\b = {\bf 1}$ and $\b = {\bf k}$ for these cases, one has 
\begin{eqnarray}
\log T_X^{\,\res,  {\bf 1}}(\rho)   &=& {\rm rk}(E_\rho)\, \chi(X). \label{Tres1}\\
\log T_X^{\,\res,  {\bf k}}(\rho)  &=  &\frac{1}{2} \dim(X) \, {\rm rk}(E_\rho)\, \chi(X).\label{Tres2}
\end{eqnarray}
If the cohomology is acyclic 
 $ \log T_X^{\,\z, {\bf \b}}(\rho) := \frac{1}{2}\sum _{k= 0}^m (-1)^{k+1} \b_k  \mbox{{\rm \TR}}_\z(\log \Delta_k)$ is a  smooth and topological invariant if and only if, up to a constant, $\b = {\bf 1}$ or $\b = {\bf k}$ 
and one has 
\begin{eqnarray}
\log T_X^{\,\z, {\bf 1}}(\rho)   & =&  0.\label{Tz1} \\
\log T_X^{\,\z, {\bf k}}(\rho)  &\stackrel{\mbox{{\rm {\tiny \cite{RaSi},\cite{Ch},\cite{Mu}}}}}{=}&  \log R_X(\rho).\label{Tz2}
\end{eqnarray}
\end{thm}

\vtwo
Thus, any topological residue torsion has the form for some $\la,\mu\in\R$
$$\log T_X^{\,\res, \,\la{\bf 1} +\mu {\bf k}}(\rho) := \frac{1}{2}\sum _{k= 0}^n (-1)^{k+1} (\la + \mu k)
\, {\rm res}(\log \Delta_k).$$
For example, on a closed surface of genus $g$, $\, \frac{1}{2}{\rm res}(\log \Delta_0) - {\rm res}(\log \Delta_1) + \frac{3}{2}  {\rm res}(\log \Delta_2)$ is a topological torsion equal to $(4g-4)\,\rk(E_\rho)$.

That it is only in the cases $\b = {\bf 1}$ or ${\bf k}$ that residue and analytic torsion define smooth invariants is obtained by a variational computation  along the lines of  \cite{RaSi}, whilst the refinement to these being topological invariants is consequent on \eqref{Tres1},  \eqref{Tres2}, \eqref{Tz1}, \eqref{Tz2}. The extension to  non-acyclic cohomology for the relative analytic torsion is straightforward, but for simplicity here we restrict attention to the acyclic case.

Included for completeness, \eqref{Tz2} is the Ray-Singer-Cheeger-Muller theorem \cite{RaSi}, \cite{Ch}, \cite{Mu} - in particular, analytic torsion is zero when $X$ is even dimensional  \cite{RaSi}, 
while in complementary fashion it is when $X$ is odd-dimensional that the residue torsion is zero.  
 If $X$ is even-dimensional $\log T_{X}^{\,\res,  {\bf k}}(\rho)$ equates to the derived Euler characteristic
$\chi{\,\pr} (X, E_\rho) : = \sum_{k=0}^n (-1)^k \,k\,\dim H^k(X,E_\rho),$
which coincides with the right-hand side of \eqref{Tres2}
and occurs in Bismut and Lott's higher analytic torsion formulae \cite{BiLo}. 
This does not quite account for all tracial analytic torsions, insofar as any trace  (proper - not a quasi-trace) on logarithm operators is a linear combination of the residue trace and a leading symbol trace \cite{PaLe}. However, since the leading symbol of  $\Delta_k$ restricts on the Riemannian cosphere bundle $S^*X$ to  the identity operator, the generalised torsion defined by any  leading symbol trace vanishes. 

\thmref{thm one} may be seen as concretising well-known formal similarities that exist between the Euler characteristic and Reidemeister/analytic torsion, in particular  the pasting formulae with respect to partitioning the manifold $X$. Just as for analytic torsion, residue torsion extends readily to manifolds with boundary: 

\begin{thm}\label{thm two}
Let $X$ be a compact manifold with boundary $Y$. Let $B$ be either relative ($\mathcal{R}$) or absolute ($\mathcal{A}$) boundary conditions for the Laplacian on $X$. The   residue torsion $T_{X, B}^{\res, \beta}(\rho)$  is a smooth and topological invariant of $(X,E_\rho)$ if and only if  $\b = {\bf 1}$ or $\b = {\bf k}$.  One has 
\begin{eqnarray}
\log T_{X,B}^{\text{res}, {\bf 1}}(\rho) &=& {\rm rk}(E_\rho)\, \chi_B(X).\label{leading trace torsion}\\
\log T_{X,B}^{\res, {\bf k}}(\rho) &=  & \chi_B'(X,E_\rho) + \sum_{k=0}^n (-1)^k k \zeta_{k,B}(0) \label{leading trace torsion weighted}\\ 
& \stackrel{\mbox{{\rm \tiny {\cite{Vi}}}}}{=} & \frac{1}{2} \dim(X) \, {\rm rk}(E_\rho)\, \chi_B(X),\label{leading trace torsion weighted 2}
\end{eqnarray}
where $R_{X,B}(\rho)$ is Reidemeister torsion and $\chi_B$ and $\chi_B'$ are the ordinary and derived Euler characteristics for  $B$.  

If the cohomology is acyclic then the analytic torsions  $T_{X,\mathcal{A}}^{\,\z, \beta}(\rho)$  and $T_{X,\mathcal{R}}^{\,\z, \beta}(\rho)$ are smooth invariants, and $T_{X,\mathcal{A}}^{\,\z, \beta}(\rho)$  is a topological invariant and $T_{X,\mathcal{R}}^{\,\z, \beta}(\rho)$     is a  PL-invariant, of $(X,E_\rho)$ if and only if  $\b = {\bf 1}$ or $\b = {\bf k}$ and one has 
\begin{eqnarray}
\log T_{X,B}^{\,\z, {\bf 1}}(\rho)   & =&  0\label{antorsionunweighted}. \\
\log T_{X,B}^{\,\z, {\bf k}}(\rho)  &\stackrel{\mbox{{\rm \tiny {\cite{Vi}}}}}{=}&  \log R_{X,B}(\rho)+\frac{1}{2}\chi(Y) \log 2.\label{antorsionvishik}
\end{eqnarray}
\end{thm}

Included again for completeness,  \eqref{antorsionvishik} is  Vishik's extension \cite{Vi} of the Ray-Singer-Cheeger-Muller theorem to manifolds with boundary. 

An exact pasting formula  for the residue torsion of a partitioned compact manifold with boundary into codimension zero submanifolds with boundary is given in $\S 2$.

{\it Acknowledgements}: The idea to look into residue torsion was originally suggested to the second author by Krzysztof Wojciechowski.
 
\section{Proof of \thmref{thm one}} 

For the Laplacian $\Delta_k:  \O^k(X ,E_\rho)\to  \O^k(X ,E_\rho)$ the 
 holomorphic functional calculus constructs the complex powers $\Delta_k^z,\,z\in\C$ \cite{Se}, and hence the pseudodifferential logarithm operator \cite{Gr, KoVi, Ok} 
\begin{equation}\label{log A}
\log\Delta_k:= \l.\partial_z(\Delta_k^z)\r|_{z=0} = \lim_{s\to 0} \frac{i}{2\pi}\int_\Cc \la\si \,\log \la \,(\Delta_k - \la)\ii \ d\la
\end{equation}
where $\la\si$ and $\log\la$ are the principal values, over a contour $\Cc$ in $\C\backslash \ \ol{\R}_-$ going around positively the non-zero  eigenvalues of $\Delta_k$. The operator \eqref{log A} has local symbol asymptotics 
$(\log \s)(x,\xi) \sim \sum_{j\geq 0} (\log \s)_{-j}(x,\xi)$ with 
\begin{equation}\label{log symbol -j}
(\log \s)_{-j}(x,\xi) := \frac{i}{2\pi}\int_{\Cc_0}  \log\la\ \rs_{-2- j}(x, \xi, \la)\ d\la
 \end{equation} 
where $\rs(x, \xi, \la)\sim \sum_{j\geq 0} \rs_{-m- j}(x, \xi, \la)$  is the parameter quasi-polyhomogeneous expansion of the symbol of the resolvent $(\D_k - \la)\ii$ for $\la$ in a closed contour  $\Cc_0$ enclosing the origin. It follows that  $(\log \s)_{-j}(x,\xi)$ is homogeneous of degree $-j$ in $|\xi|\geq 1$ for $j\geq 1$, and has leading symbol
$(\log \s)_0(x,\xi)  = 2\,\log |\xi|  +   \log \l(\frac{|\xi|_{g(x)}}{|\xi|}\r)$. 
The second summand of $(\log \s)(x,\xi)$ is homogeneous of degree zero, and hence  $\log \D_k$ is a `logarithmic $\pdo$', in the sense of \cite{Ok} with a well-defined   residue trace character 
\begin{equation}\label{resdetA}
\res(\log \D_k)  =
\frac{1}{(2\pi)^n}\int_M\int_{|\eta|=1}\tr\left((\log \s)_{-n}(x,\eta)\right)
\ d_{\mbox{{\tiny S}}}\eta\, dx.
\end{equation}
If $n=\dim X$ is odd each of these numbers is identically zero, and hence so is $\log T_X^{\,\res, {\bf \b}}(\rho)$ - for any $\b\in\R^{n+1}$. To see this, since $(\D_p - \la)\ii$ has scalar-valued leading symbol  $\rs_{-2}(x, \xi, \la) = (|\xi|^2_{g(x)} - \la)\ii$ one has the resolvent symbol formula over $U\<X$ 
$$\rs_{-2-n}(x, \xi, \la) = \sum_{\frac{n}{2}\leq k \leq 2n} \rs_{-2}(x, \xi, \la)^{k+1} \rs_{-2-n,k}(x,\xi)$$
in which $\rs_{-2-n,k} \in C^\oo (U\x\R^n, \End \R^N)$ is a finite product of derivatives (in $x$ and $\xi$) of the symbol components of $\D_p$ in the local trivialization over $U$, and therefore polynomial in $\xi$ and independent of $\la$. If $n$ is an odd integer it is an odd function of  $\xi$ of degree $2k-n$, i.e. $\rs_{-2-n,k}(x,-\xi) = (-1)^{2k-n}\rs_{-2-n,k}(x,\xi) = -\rs_{-2-n,k}(x,\xi)$, and since 
\begin{equation*}\label{log symbol -n res}
\tr_x\l((\log \s(\D_p))_{-n}(x,\xi)\r) := \frac{i}{2\pi}\int_{\Cc_0}  \log \la\ \tr_x(\rs_{-2-n}(x, \xi, \la))\ d\la
 \end{equation*} 
then $\tr_x\l((\log \s(\D_p))_{-n}(x,-\xi)\r) = - \tr_x\l((\log \s(\D_p))_{-n}(x,\xi)\r) $ and so
\begin{equation}\label{odd log res}
\int_{S^{n-1}} \tr_x\l((\log \s(\D_p))_{-n}(x,\eta)\r)  d_{\mbox{{\tiny S}}}\eta= 0.
\end{equation} 
%since for $f$ an odd function with  $n$ odd   $\int_{S^{n-1}} f(\eta) \ d_{\mbox{{\tiny S}}}\eta   = -\int_{S^{n-1}} f(-\eta) \ d_{\mbox{{\tiny S}}}\eta  = - \int_{S^{n-1}}  f(\eta) \ d_{\mbox{{\tiny S}}}\eta$. 
Thus, $\log \D_p$ is `odd class' (`even-even') \cite{KoVi, Ok, Gr} and $\res(\log \D_p)=0$.

To determine for which ${\bf \b}\in\R^{n+1}$ smooth invariance of the torsions holds, consider a smooth path of metrics $u \in \mathbb{R} \mapsto g^X(u)$, defining smooth paths of Hodge operators $* = *_u$  and Laplacians $\Delta_k= \Delta_k(u)$ on $k$ forms.  Consider, first, residue torsion. In view of the ellipticity of the Laplacian, 
\begin{align*}
\frac{d}{du} \res (\log \Delta_k) =   \text{res} \left( \dot{\Delta}_k^{}P_k  \right)= \text{res} \left(P_k \dot{\Delta}_k^{} \right)
\end{align*}
for any parametrix $P_k$ for $\Delta_k$; that is,  $P_k \in  \Psi^{-2}(X, \Lambda T^k X \otimes E_\rho)$ with  $P_k \Delta_k - I$ and $I - \Delta_k P_k$  in $\Psi^{-\infty}:= \Psi^{-\infty}(X, \Lambda(X) \otimes E_\rho)$
(smoothing). We may  choose $P_k:=(\Delta_k + \Pi_k)^{-1}$ 
with $\Pi_k$ the orthogonal projection onto $\ker \Delta_k\cong H^k(X)$. By ellipticity, $\Pi_k$ is  finite rank and independent of the metric, $\Delta_k +\Pi_k$ is a smooth family of invertible operators with $\frac{d}{du}(\Delta_k+ \Pi_k)= \dot{\Delta}_k$ and $
I=P_k(\Delta_k^{} + \Pi_k)=P_k \Delta_k^{} + P_k\Pi_k$ with $P_k\Pi_k\in \Psi^{-\infty}.$
 
Since $*^{-1}_k=(-1)^{k}*^{}_{n-k}$, then $\alpha_k:=*_k^{-1}\ \dot{*}_k^{}= -\ \dot{*}_{n-k}^{}\ *_{n-k}^{-1}: \Lambda^k(X, E_\rho) \to \Lambda^k(X, E_\rho)$, and  similarly to Theorem 2.1 of \cite{RaSi} with $\delta_k:=d_k^*$
\begin{eqnarray*}\label{derivativeoflaplaciano}
\dot{\Delta}_k&  = & \frac{d}{du}(\delta_k d_k + d_{k-1} \delta_{k-1})   =   - \dot{*} \ d * d - *\  d\  \dot{*}\  d -  d\ \dot{*}\ d *-  d * d\ \dot{*}  \\
 % & =  & - \dot{*} \ *^{-1} *\ d * d - *\  d\ *\ *^{-1} \dot{*}\  d -  d\ \dot{*}\ *^{-1} *\ d *-\  d * d\ *\ *^{-1} \dot{*}  \\  
& =   & + \alpha *\ d * d - *\  d\ *\ \alpha\  d +  d\ \alpha *\ d *-  d * d\ *\ \alpha  \\ 
& =   & -\alpha_{k}\delta_{k} d_{k} + \delta_{k} \alpha_{k+1}  d_{k} -  d_{k-1} \alpha_{k-1} \delta_{k-1} + d_{k-1} \delta_{k-1} \alpha_{k}.
\end{eqnarray*}
Hence $\text{res}(P_k\dot{\Delta}_k^{}) $ is equal to 
\begin{eqnarray*}
   \underbrace{- \text{res}(P_k\alpha_{k}^{}\delta_{k}^{} d_{k}^{})}_{(i)}  \ + \ \underbrace{\text{res}(P_k\delta_{k}^{} \alpha_{k+1}^{}  d_{k}^{})}_{(ii)} 
\ - \ \underbrace{ \text{res}(P_kd_{k-1}^{} \alpha_{k-1}^{} \delta_{k-1}^{})}_{(iii)} \ + \ \underbrace{\text{res}(P_kd_{k-1}^{} \delta_{k-1}^{} \alpha_{k}^{})}_{(iv)}.
\end{eqnarray*} 
The following identities hold, the first two exactly, and the second two modulo $\Psi^{-\oo}$
\begin{equation}\label{identitiesforlaplacian}
d_k \Delta_k^{} = \Delta_{k+1} d_k  \ \ \    \delta_{k-1} \Delta_{k}^{}=\Delta_{k-1}^{} \delta_{k-1}, \ \ \ 
d_k P_k= P_{k+1} d_k  \ \ \    \delta_{k-1} P_k =P_{k-1} \delta_{k-1}.
\end{equation}
For, since  $\Delta_k^{}P_k - I  \in \Psi^{-\infty}$ then the difference of $d_k \Delta_k^{}P_k - d_k$ and $\Delta_{k+1}^{} P_{k+1}d_k - d_k$ is smoothing, and hence that  $\Delta_{k+1}^{} (d_kP_k - P_{k+1}d_k) \in \Psi^{-\infty}$ implying the third equality of \eqref{identitiesforlaplacian} by ellipticity, and likewise  $\delta_{k-1} P_k-P_{k-1} \delta_{k-1} \in \Psi^{-\infty}$. Hence
$$
{(i)}  = - \text{res}(P_k\alpha_{k}^{}\delta_{k}^{} d_{k}^{})  =  - \text{res}(\delta_{k}^{} d_{k}^{} P_k\alpha_{k}^{}) = - \text{res}(\delta_{k}^{} P_{k+1} d_{k}^{} \alpha_{k}^{}) 
 = - \text{res}(P_k \delta_{k}^{} d_{k}^{} \alpha_{k}^{}),$$
$$ {(iii)}  =-  \text{res}(P_kd_{k-1}^{} \alpha_{k-1}^{} \delta_{k-1}^{})= -  \text{res}(\delta_{k-1}^{} P_kd_{k-1}^{} \alpha_{k-1}^{} )
 =  -  \text{res}(P_{k-1}\delta_{k-1}^{} d_{k-1}^{} \alpha_{k-1}^{}).$$
On the other hand, since $\alpha_{k}^{} - P_k\delta_{k}^{} d_{k}^{} \alpha_{k}^{} - P_k d_{k-1}^{} \delta_{k-1}^{} \alpha_{k}^{}= \alpha_{k}^{} -P_k\Delta_k^{} \alpha_{k}^{} \in \Psi^{-\infty}$, then 
\begin{align*}
(ii) & =\text{res}(P_k\delta_{k} \alpha_{k+1}  d_{k}) =\text{res}(d_{k}P_k\delta_{k} \alpha_{k+1})=\text{res}(P_{k+1}d_{k}\delta_{k} \alpha_{k+1}) \\
&= \text{res}(\alpha_{k+1}^{}) - \text{res}(P_{k+1}\delta_{k+1}^{} d_{k+1}^{}\alpha_{k+1}^{}), \\
(iv) &=\text{res}(P_kd_{k-1}^{} \delta_{k-1}^{} \alpha_{k}^{})= \text{res}(\alpha_k) - \text{res}(P_k\delta_{k}^{} d_{k}^{}\alpha_{k}^{}).
\end{align*}
Setting $$\gamma_k^{} := \begin{cases}
\text{res}(P_k\delta_{k}^{} d_{k}^{}\alpha_{k}^{}) & k\in\{0,1,\ldots,n\},\\
0  & k\in\Z\bsh\{0,1,\ldots,n\},
\end{cases}$$
 we have
${(i)}  = - \gamma_{k}^{},  (ii)  = \text{res}(\alpha_{k+1}^{}) - \gamma_{k+1}^{},  {(iii)}  =- \gamma_{k-1}^{},$  $(iv)= \text{res}(\alpha_k) - \gamma_k^{},$ and
 $$\text{res}(P_k\dot{\Delta}_k^{}) = \text{res}(\alpha_k) + \text{res}(\alpha_{k+1}^{}) - \gamma_{k+1}^{}- 2\gamma_k^{} - \gamma_{k-1}^{}  = - \gamma_{k+1}^{}- 2\gamma_k^{} - \gamma_{k-1}^{},$$ as $\res(\alpha_k)=0$ since $\alpha_k \in \text{End}(\Lambda^k(M, E_\rho))$ is an order zero bundle endomorphism. Hence
$$
2\frac{d}{du} \log T_X^{\text{res},\beta}(\rho, g_u)  
 =   -2 \sum_{k=0}^{n} (-1)^{k+1} \,\beta_k \gamma_k^{} - \sum_{k=0}^{n} (-1)^{k+1} \, \beta_k \gamma_{k+1}^{} -   \sum_{k=0}^{n} (-1)^{k+1} \, \beta_k \gamma_{k-1}^{}$$ 
\begin{equation}\label{calculationsaturday}
= \ \  \sum_{k=1}^{n-1} (-1)^{k+1} \ (\beta_{k+1}-2\beta_k+\beta_{k-1})  \gamma_k^{} \ + \ (\beta_1-2\beta_0)  \gamma_0 \ + \ (\beta_{n-1}-2\beta_n) \gamma_n.  
\end{equation}
The $k=0, k=n$ summands need to be treated separately as they involve $\g_{-1}:=0, \g_{n+1}:=0$. However, both summands vanish: $\gamma_0 = \text{res}(P_0\delta_0^{} d_0^{}\alpha_0^{}) = \text{res}(\alpha_0^{})=0$ and likewise for $\g_n$. We are thus left with just the summation term. For this, the $\g_k = \g_k(g_u)$ are linearly independent as elements in the space of smooth functions on the Frechét manifold of Riemannian metrics on $X$; that is, constants $\la_k$ with $$\la_1\g_1(g_u) + \cdots + \la_m\g_m(g_u) =0$$ for any metric $g = g_u$ necessarily vanish, $\la_k=0$ for all $k$, since  the numbers 
$$\g_k(g_u) = \text{res}(\Delta_k\ii d_k^*d_k \alpha_k) $$
depend explicitly on the choice of Riemannian metric, and there is no simple linear relation between them. This is seen by computations in local coordinates using the Bochner formula 
$$\Delta_k = \sum_{ij} g^{ij}\nabla_{e_i}\nabla_{e_j} - \sum_{ijl} g^{ij}\Gamma^l_{ij}\nabla_{e_l}  + D^k R, \ \ \ \ \ \  d^*_k = - *_g \ d_k *_g, $$
for an orthonormal frame $\{e_j\}$  with $\nabla_{e_j}$ the metric connection, $R$ the Riemann curvature tensor and $D^k R$ the induced exterior algebra $k$-form derivation.  On the other hand, from \eqref{calculationsaturday} smooth invariance of the residue torsion $\frac{d}{du} \log T_X^{\text{res},\beta}(\rho, g_u)  =0$  for any path of metrics $g_u$ has been reduced to 
$$\sum_{k=1}^{n-1} (-1)^{k+1} \ (\beta_{k+1}-2\beta_k+\beta_{k-1})  \gamma_k (g) =0$$
for constants $\b_j$. This can therefold hold (on arbitrary $X$) if and only if 
$$\beta_{k+1}-2\beta_k+\beta_{k-1}=0$$
for each non-negative integer $k$. Elementary methods give the solution of this recurrence relation to be a linear combination of the two independent solutions $\beta_k=1$ for all $k$, or $\beta_k=k$ for all $k$. That is, residue torsion is a smooth invariant of the Riemannian metric if and only if  $\b = \la {\bf 1} + \mu {\bf k}$ for arbitrary constants $\la,\mu$. 

As a simple example, if $X$ is a (real) closed surface then the residue torsion variation formula is
$$\frac{d}{du} \log T_X^{\text{res},\beta}(\rho, g_u)   =   - (\beta_2-2\beta_1+\beta_0) \, \text{res}(P_1\delta_1 d_1\alpha_1)$$
for $\b = (\b_0,\b_1,\b_2)\in\R^3$. Since $\text{res}(P_1\delta_1 d_1\alpha_1) $ is generically non-zero on the space $\Mm(X)$ of metrics on $X$, then the variation is zero on $\Mm(X)$ precisely for $\b$ in the plane $\beta_2-2\beta_1+\beta_0=0$ in $\R^3$,  and this plane is spanned by the basis vectors ${\bf 1} = (1,1,1)$ and ${\bf k} = (0,1,2)$.

The case of analytic torsion is similar (again guided by Theorem 2.1 of \cite{RaSi}). Assuming the cohomology with coefficients in $E_\rho$ is trivial and $X$ is odd-dimensional, then
\begin{align*}
f(u,s):=\frac{1}{2}\sum_{k=0}^n(-1)^k \beta_k \int_0^\infty t^{s-1} \text{Tr}\left(e^{-t \Delta_k}\right) dt
\end{align*}
is well-defined for $\Re(s) \gg 0$ and extends meromorphically to $\C$  with no pole at $s=0$ with $$f(u,0)= \log T_X^{\,\zeta,\beta} (\rho).$$ For smooth invariance we wish to examine when $\frac{\partial}{\partial u}f(u,0)=0$. Since we have the uniform estimate
$$ \text{Tr}\left(e^{-t \Delta_k}\right) \leq Ce^{-\epsilon t}, \ \ \ \ \ \ \ t\geq t_0>0,$$
for $C, \epsilon >0$, independently of $u$, then we can differentiate under the integral to get for $\Re(s)$ large
$$
\frac{\partial}{\partial u}f(u,s)=\frac{1}{2}\sum_{k=0}^{n}(-1)^{k+1} \beta_k \int_0^\infty t^{s}\, \text{Tr}\left(e^{-t \Delta_k}\dot{\Delta}_k\right) dt.
$$
Setting $\varphi_k:=\text{Tr}\left(e^{-t \Delta_{k}}d\delta\alpha   \right)$ and $\theta_k:=\text{Tr}\left(e^{-t \Delta_{k}}\delta d\alpha  \right)$,  we can rewrite this as
\begin{align}\label{var fu}
\frac{\partial}{\partial u}f(u,s)= \int_0^\infty t^{s} \sum_{k=0}^n(-1)^{k+1} \beta_k\left(\varphi_{k+1} - \theta_k + \varphi_k - \theta_{k-1}\right) dt.
\end{align}
Noting $\varphi_0=\theta_n=0$, 
\begin{align}
&\sum_{k=0}^n(-1)^{k+1} \beta_k\left(\varphi_{k+1} - \theta_k + \varphi_k - \theta_{k-1}\right) \nonumber \\= &\sum_{k=1}^n(-1)^{k+1}\left[(\beta_k - \beta_{k-1})\varphi_k + (\beta_{k+1}-\beta_k) \theta_k\right]+ (\beta_0 - \beta_1)\theta_0  + (-1)^n(\beta_{n-1} - \beta_n)\varphi_n\nonumber\\
=&\sum_{k=1}^{n}(-1)^{k+1}(2\beta_k - \beta_{k-1}- \beta_{k+1})\varphi_k \ \ + \ \  \sum_{k=0}^{n}(-1)^{k+1}(\beta_{k+1}-\beta_k)  \text{Tr}\left(e^{-t \Delta_{k}}\Delta_k\alpha   \right)\nonumber\\
& = \ \ \sum_{k=1}^{n-1}(-1)^{k+1}(2\beta_k - \beta_{k-1}- \beta_{k+1})\varphi_k \ \  +\ \  (2\beta_n -\beta_{n-1})\varphi_n \label{var1}\\
 & \hskip 7mm + \sum_{k=0}^{n-1}(-1)^{k}(\beta_{k+1}-\beta_k)  \frac{d}{dt}\text{Tr}\left(e^{-t \Delta_{k}}\alpha   \right) \ \ + \ \ \beta_n \frac{d}{dt}\text{Tr}\left(e^{-t \Delta_n}\alpha   \right)  \label{var2}
\end{align}
with the $k=n$ term of the sums separated off corresponding to the vanishing $\varphi_{n+1}=0$. 
Each of the summands in \eqref{var2} contributes to \eqref{var fu}, via integration by parts, a factor
$ s\int_0^\infty t^{s-1} \text{Tr}\left(e^{-t \Delta_{k}}\alpha   \right) dt$ -- but this vanishes at $s=0$, since, exactly as in  Thm 2.1 of \cite{RaSi},
the integral is holomorphic at 0 in view of $\res(\alpha)=0$. Thus it is only the terms in  \eqref{var1}  which may contribute to  
$\partial_u f(u,s)$. First, by the same argument, since $\varphi_n =\frac{d}{dt}\text{Tr}\left(e^{-t \Delta_{n}}\alpha\right)$ its contribution to the variation of $f(u,s)$ is zero. So we are left with
\begin{align*} 
\frac{\partial}{\partial u}f(u,s)&=\frac{1}{2}\sum_{k=1}^{n}(-1)^{k}(\beta_{k+1} - 2\beta_k + \beta_{k-1})  \int_0^\infty t^{s} \varphi_k \,dt \\
&= \frac{1}{2}\sum_{k=1}^{n}(-1)^{k}(\beta_{k+1} - 2\beta_k + \beta_{k-1})  \Gamma(s+1) \zeta(\Delta^{-1}_kd \delta \alpha_k, \Delta_k, s).
\end{align*}
The $s$ dependent terms are  holomorphic at $s=0$, since $\res(\Delta^{-1}_kd \delta \alpha_k)=0$ as $\Delta^{-1}_kd \delta \alpha_k$ is odd class, giving on setting $s=0$  the analytic torsion  variation formula 
$$
\frac{\partial}{\partial u}\log T_X^{\,\zeta,\beta} (\rho) = \frac{1}{2}  \sum_{k=1}^{n}(-1)^{k}(\beta_{k+1} - 2\beta_k + \beta_{k-1})\zeta(\Delta^{-1}_kd \delta \alpha_k, \Delta_k, 0).
$$
As for residue torsion, no non-zero linear relation exists between the metric dependent terms $\zeta(\Delta^{-1}_kd \delta \alpha_k, \Delta_k, 0)$. The variation is hence zero precisely when $\beta_{k+1}-2\beta_k+\beta_{k-1}=0$ for each $k$. $T_X^{\,\zeta,\beta} (\rho)$ is hence a smooth invariant only when $\b =   {\bf 1}$ or ${\bf k}$, the latter case being \cite{RaSi}.

For both residue and analytic torsion the above computations can be repeated for a variation of the metric on $E_\rho$, but the details are similar. 

\vtwo
 
To prove the exact formulae \eqref{Tz1}, \eqref{Tres1} and \eqref{Tres2}, let $ \zeta_{k, \rho}(s)  = \Tr (\D_k^s)$ for the Laplacian $\D_k  :  \O^k(X ,E_\rho)\to  \O^k(X ,E_\rho)$. 
For each $k=0, \dots, n$,  $\Delta_k$ and $\Delta_{n-k}$ are isospectral, since $*_k\Delta_k=\Delta_{n-k}*_k$. Therefore we obtain the Poincar\'e Duality property
\begin{align}\label{zetasymmetric}	
\zeta_{k, \rho}(s) =\zeta_{n-k, \rho}(s)
\end{align}
$\forall s\in \mathbb{C}$ by uniqueness of continuation.
If $n$ is even  $(-1)^{n-k} = (-1)^k$ and 
\begin{align*}
\sum_{k=0}^n (-1)^{k} k \zeta_{k, \rho}(s) \stackrel{(\ref{zetasymmetric})}{=}  &\sum_{k=0}^n (-1)^{k} k \zeta_{n-k, \rho}(s) = \sum_{k=0}^n (-1)^{n-k} (n-k) \zeta_{k, \rho}(s) \\
 = &\sum_{k=0}^n (-1)^{k} (n-k) \zeta_{k, \rho}(s) = n\sum_{k=0}^n (-1)^k  \zeta_{k, \rho}(s) - \sum_{k=0}^n (-1)^{k}k \zeta_{k, \rho}(s).
\end{align*}
Hence, for even $n$ 
 \begin{equation}\label{i}
 \frac{n}{2}\sum_{k=0}^n (-1)^{k} \ \zeta_{k, \rho}(s) = \sum_{k=0}^n (-1)^{k} k\ \zeta_{k, \rho}(s),
 \end{equation}
 while  if $n$ is odd, then $(-1)^{n -k}= - (-1)^k$ and
\begin{align*}
\sum_{k=0}^n (-1)^{k} \ \zeta_{k, \rho}(s)  \stackrel{(\ref{zetasymmetric})}{=}\sum_{k=0}^n (-1)^{k} \ \zeta_{n-k, \rho}(s) = \sum_{k=0}^n (-1)^{n-k} \ \zeta_{k, \rho}(s)= - \sum_{k=0}^n (-1)^{k} \ \zeta_{k, \rho}(s).
\end{align*}
So that
 \begin{equation}\label{ii}
\sum_{k=0}^n (-1)^{k} \ \zeta_{k, \rho}(s) = 0  \ \ \  \ffor \ n =\dim X \ {\rm odd}.
 \end{equation}
From  \cite{RaSi},  Theorem 2.3, we recall that  if $n$ is even
\begin{equation}\label{zeta vanishing k}
\sum_{k=0}^n (-1)^{k} k\ \zeta_{k, \rho}(s) = 0,
\end{equation}
and hence from \eqref{i}
\begin{equation}\label{zeta vanishing 1}
\sum_{k=0}^n (-1)^{k} \ \zeta_{k, \rho}(s) =0   \ \ \  \ffor \ n =\dim X \ {\rm even}.
\end{equation}
\eqref{ii} and \eqref{zeta vanishing 1} may together equivalently be seen by identifying $\sum_{k=0}^n (-1)^{k} \ \zeta_{k, \rho}(s)$ with the difference of spectral zeta functions $\zeta(D^*D,s) - \zeta(D^*D,s)$ with $D:= \bigoplus_k d_{2k} + d_{2k}^*$ acting on  forms of even degree, with range in forms of odd degree, noting that  $D^*D$ and $DD^*$ have identical non-zero eigenvalue spectrum.   Either way, this gives
\begin{equation*}
\log T_X^{\,\z, {\bf 1}}(\rho)   =  0 \ \ \ \forall \ n =\dim X.
\end{equation*}
Consider $$\chi(X, E_\rho):=\sum_{k=0}^n (-1)^k \,\dim H^k(X,E_\rho)\ \ \ \aand \ \ \ 
\chi'(X, E_\rho):=\sum_{k=0}^n (-1)^k k \,\dim H^k(X,E_\rho).$$
Set $b_k:= \dim H^k(X,E_\rho)$. By Poincar\'e duality
\begin{align*}
\chi'(X, E_\rho)&=\sum_{k=0}^{n}(-1)^k k b_k=\sum_{k=0}^{n}(-1)^k k b_{n-k}=\sum_{k=0}^{n}(-1)^{n-k} (n-k) b_k\\
&=(-1)^n n\sum_{k=0}^{n}(-1)^k b_k + (-1)^{n-1} \sum_{k=0}^{n}(-1)^k kb_k\\
&= (-1)^{n-1} \chi'(X, E_\rho) + (-1)^n n \chi(X, E_\rho).
\end{align*}
Hence
\begin{equation}\label{weighedcasenotveryinteresting}
\chi'(X, E_\rho)(1 + (-1)^n)=n\chi(X, E_\rho).
\end{equation}
So if $n$ is even
\begin{equation}\label{corollarinoinoino}
\chi'(X, E_\rho)=\frac{n}{2}\chi(X, E_\rho).
\end{equation}
$\chi'(X, E_\rho)$ thus does not  provide new information if $\dim X$ is even, but it may when $\dim X$ is odd --- in general,  the $j^{{\rm th}}$ derived Euler characteristic $\chi_j(X, E_\rho)$ is the first nontrivial homotopy invariant when $\chi_k(X, E_\rho)$ vanishes for each $k <j$ \cite{Rm}.

The residue determinant is related to the  spectral zeta function $\z_{k, \rho}(z) := \l.\Tr\l(\D_k^z\r)\r|^{{\rm
mer}} $  by   \cite{Sc2}
\begin{equation}\label{resdetzetazero}
-\frac{1}{2}{\rm res}(\log \D_k)  =  \z_{k, \rho}(0) \, +  \dim\ker (\D_k).
\end{equation}
Let $n=\dim X$ be even. Then
\begin{eqnarray*}
\log T_X^{\res, {\bf 1}}(\rho) &\stackrel{\eqref{resdetzetazero}}{=} & \sum_{k=0}^{n} (-1)^k \zeta_{k, \rho}(0) + \sum_{k=0}^{n} (-1)^k \dim \ker(\Delta_k^{}) \\ &\stackrel{\eqref{zeta vanishing 1} }{=} &\sum_{k=0}^{n} (-1)^k \dim \ker(\Delta_k^{})  \\[1mm]&= & \chi(X,E_\rho)  \ =  \ \rk(E_\rho)\,\chi(X)
\end{eqnarray*}
using   the Hodge theorem for the third equality, and the index theorem for the fourth. Similarly,  if $\beta= {\bf k}$
\begin{eqnarray*}
\log T_X^{\res, {\bf k}}(\rho) &=& \sum_{k=0}^{n} (-1)^k k\zeta_{k, \rho}(0) + \sum_{k=0}^{n} (-1)^k k\dim \ker(\Delta_k^{}) \\[1mm]&\stackrel{\eqref{zeta vanishing k}}{=}
 & \chi^\prime (X,E_\rho)  \ \stackrel{\eqref{corollarinoinoino}
}{=} \ \frac{n}{2}\,\rk(E_\rho)\,\chi(X).
\end{eqnarray*}

In summary: \eqref{Tz1}, \eqref{Tz2} (by \cite{RaSi},\cite{Ch},\cite{Mu}), \eqref{Tres1},  \eqref{Tres2} hold and are topological invariants (homotopy invariants for the latter two) whilst, conversely, if $T_X^{\res, \beta}(\rho)$ and $T_X^{\zeta, \beta}(\rho)$ are topological invariants then they are smooth invariants and this implies $\beta= {\bf 1}$ or ${\bf k}$. This completes the proof of \thmref{thm one}.

\section{Proof of \thmref{thm two}}

When $X$ is an $n$-dimensional smooth manifold with non-empty  boundary $Y$, we assume it embedded into a closed $n$-dimensional manifold $\widetilde{X}$ and with a product structure on a collar neighbourhood $U\cong [0, c) \times Y$. Thus, we have an orthogonal decomposition of smooth $k$-forms 
$\omega_{|U}=\omega_1 + dt  \wedge\omega_2$, 
where $\omega_1 \in C^\infty([0,c)) \otimes \Omega^k(Y, {E_\rho}_{|Y})$ and $\omega_2 \in C^\infty([0,c))  \otimes\Omega^{k-1}(Y, {E_\rho}_{|Y})$, which yields the orthogonal projections (\cite{Sch})
\begin{align*}
\begin{array}{cc}
\mathcal{R}: \Omega(X, E_\rho)_{|Y}  \to \Omega(Y, {E_\rho}_{|Y}) & \qquad \mathcal{A}: \Omega(X, E_\rho)_{|Y} \to \Omega(Y, {E_\rho}_{|Y})\\
\omega_{|Y}  \mapsto \omega_1 &\qquad \omega_{|Y} \mapsto \omega_2
\end{array}
\end{align*}
 $\Omega^k(X, E_\rho)_{|Y}$ as  the space of
boundary restrictions of smooth $k$-forms. Since Green's formula yields ((2.8), \cite{DaFa}):
\begin{align}\label{absolrelboundcondforlaplacian}
\langle \Delta_k \omega, \theta\rangle_X= \langle \omega,  \Delta_k\theta\rangle_X+\int_Y \omega \wedge * d \theta - \int_Y \theta \wedge * d \omega+\int_Y \delta \omega \wedge * \theta- \int_Y\delta \theta \wedge * \omega,
\end{align}
complex powers and a (pseudodifferential) logarithm are defined from {$\Delta_k: \Omega^k(X, E_\rho) \to \Omega^k(X, E_\rho)$} endowed with, as in \cite{DaFa} (\S 2.1),  relative or absolute boundary conditions 
\begin{align*}
\text{Relative: }\begin{cases}
\mathcal{R}\omega_{|Y}=0 & \\
 \mathcal{R}  \delta \omega_{|Y}=0 & 
\end{cases} \qquad \text{Absolute: }\begin{cases}\mathcal{A}\omega_{|Y}=0 & \\
 \mathcal{A}  d \omega_{|Y}=0 & 
\end{cases}
\end{align*}

From (\ref{absolrelboundcondforlaplacian}),  the realization $\Delta_{k,B}$ for $B=\mathcal{R}$ or $B=\mathcal{A}$ conditions is self-adjoint and has a discrete set of non-negative eigenvalues accumulating at infinity, with a corresponding orthonormal basis of eigenvalues for $L^2(X,E_\rho)$ satisfying the boundary conditions (\S 3.3, \cite{Gr2}). 
The holomorphic family 
\begin{align*}
\Delta_{k,B}^{-s} := \frac{i}{2\pi}\int_{\Cc}   \lambda^{-s}\ (\Delta_{k,B} - \lambda)^{-1}d\lambda, \qquad \Re(s)>0,
\end{align*}
over a contour $\Cc$ in $\C\backslash \ \ol{\R}_-$  positively around the non-zero  eigenvalues of $\Delta_{k,B}$,
was defined by Seeley \cite{See1}, \cite{See2}, and is trace class for $\Re(s)>n/2$  \cite{See1}, while a logarithm
\begin{align}\label{logarithmoftherealization}
\log \Delta_{k,B} := \lim_{s \searrow 0} \frac{i}{2\pi}\int_C   \lambda^{-s}\log \lambda\ (\Delta_{k,B} - \lambda)^{-1}d\lambda
\end{align}
was studied by Grubb and Gaarde ((2.5) in \cite{GaGr}). (\ref{logarithmoftherealization})  equals $   (\log \widetilde{\Delta}_k)_+ + G^\text{log}$, 
where $(\log \widetilde{\Delta}_k)_+$ is the restriction to $X$ of the (classical) logarithm of the Laplacian on $k$-forms on $\widetilde{X}$ and $G^\text{log}$ is a singular Green operator. The local symbol of $(\log \widetilde{\Delta}_k)_+$ is analogous to (\ref{log symbol -j}), while $G^\text{log}$ has symbol-kernel of quasi-homogeneous terms satisfying part of the usual estimates for singular Green operators (Theorem   2.4, \cite{GaGr}).

As for the closed case, $\zeta_{k, B}(s):=\text{Tr}(\Delta_{k,B}^{-s})|^\text{mer}$, with Tr the classical trace,  is holomorphic at zero and thus is used to define the analytic torsion with absolute/relative boundary conditions 
\begin{align*}
\log T^{\z, \bf k}_{X, B}(\rho) = \frac{1}{2} \sum_{k=0}^n (-1)^{k} k\, \zeta_{\,k, B}'(0), \qquad B=\mathcal{R} \text{ or } \mathcal{A}.
\end{align*}
Since $\text{Tr}(\Delta_{k,B}^{-s})$ is holomorphic for $\Re(s)>n/2$, 
$\frac{d}{ds}\text{Tr}(\Delta_{k,B}^{-s})=\text{Tr}(\frac{d}{ds}\Delta_{k,B}^{-s})$ for such $s$. So, by expansion (1.12) in \cite{GrSc}, both $\text{Tr}(\Delta_{k,B}^{-s})$ and $\text{Tr}(\frac{d}{ds}\Delta_{k,B}^{-s})$ can be extended meromorphically and are holomorphic at $s=0$, which yields 
$\frac{d}{ds}\zeta_{\Delta_{k,B}}^{}(0)=- \TR_\z(\log\Delta_{k,B})$ as for the boundaryless case. Hence
\begin{align*}
\log T^{\z, \bf k}_{X, B}(\rho) = \frac{1}{2} \sum_{k=0}^n (-1)^{k+1} k \,\TR_\z(\log\Delta_{k,B}), \qquad B=\mathcal{R} \text{ or } \mathcal{A}.
\end{align*}
Similarly, $\log T^{\z, \bf1}_{X,B}(\rho):= \frac{1}{2} \sum_{k=0}^n (-1)^k \,  \zeta^\prime_{k, B} (0)=\frac{1}{2} \sum_{k=0}^n (-1)^k \,\TR_\z(\log\Delta_{k,B})$.

 $\log \Delta_{k,B}$ belongs to the Boutet the Monvel calculus \cite{GaGr}. There, the residue trace has been extended by work of Fedosov, Golse, Leichtnam, and Schrohe \cite{FeGoLeSc} and is the unique trace.  Hence, we have a well-defined $\res(\log \Delta_{k,B})$, which we can use to define a  (generalized) residue analytic torsion of $X$ with either relative or absolute boundary conditions
\begin{align}\label{residuetorsionwhenwehaveboundary}
\log T^{\res, \beta}_{X, B}(\rho)=\frac{1}{2} \sum_{k=0}^n (-1)^{k+1} \beta_k\ \res (\log \Delta_{k,B}).
\end{align}
To see that $\log T_{X,B}^{\res, \b}(\rho)$ is independent of the Riemannian metric  if and only if  $\beta$ equals
${\bf 1} :=(1, \dots, 1)$  or ${\bf k} := (0,1, \dots, n)$ we have, similarly to the boundaryless case, from \cite{Gr} that
\begin{align}\label{ScottsformulabutnowGrubb}
- \frac{1}{2} \res (\log \Delta_{k,B})= \zeta_{k, B}(0) + \dim \ker (\Delta_{k,B}),
\end{align}
as relative/absolute boundary conditions are normal, and so
\begin{align*}
\log T^{\res, \beta}_{X, B}(\rho)= \sum_{k=0}^n (-1)^{k} \beta_k\ \zeta_{k, B}(0) + \sum_{k=0}^n (-1)^{k} \beta_k\ \dim \ker (\Delta_{k,B}).
\end{align*}
Let $u\in [0,1]  \mapsto g^X(u)$ be  a smooth path of metrics  for which the normal direction to the boundary $Y$ is the same and consider $\frac{d}{du}\log T^{\res, \beta}_{X, B}(\rho)$. Since $\ker (\Delta_{k,B})$ is isomorphic to relative/absolute de Rham cohomology, it is independent of the metric (see for instance the proof of Proposition 6.4, \cite{RaSi}, or (2.5) in \cite{Vi}) and the derivative reduces to
\begin{align*}
\frac{d}{du}\log T^{\res, \beta}_{X, B}(\rho)= \frac{d}{du}\sum_{k=0}^n (-1)^{k} \beta_k\ \zeta_{k, B}(0).
\end{align*}
Without loss of generality, as we are considering the residue trace, take $\Delta_{k,B}$ to be invertible. Then  $\zeta_{k, B}(s)=\int_0^\infty t^{s-1} \text{Tr}(e^{-t \Delta_{k, B}(u)})dt$ and we can study the derivative at $s=0$ of
\begin{align*}
f(u,s):= \sum_{k=0}^n (-1)^{k} \beta_k\ \int_0^\infty t^{s-1} \text{Tr}(e^{-t \Delta_{k, B}(u)})dt.
\end{align*}
By Theorem 6.1 in \cite{RaSi} $$\frac{\partial}{\partial u}\text{Tr}(e^{-t \Delta_{k, B}})=-t\, \text{Tr}((\delta \alpha_k d -d\alpha_k\delta+\alpha_k d \delta - \alpha_k \delta d)e^{-t \Delta_{k, B}})$$ and, by the proof of Proposition 2.15 in \cite{Vi}, we can differentiate under the integral sign, thus obtaining
\begin{align*}
\frac{\partial}{\partial u} f(u,s)  =\sum_{k=0}^n (-1)^{k+1} \beta_k\ \int_0^\infty t^{s} \,\text{Tr}((\delta \alpha_k d -d\alpha_k\delta+\alpha_k d \delta - \alpha_k \delta d)e^{-t \Delta_{k, B}})dt.
\end{align*}
Moreover,  by Theorem 7.3 in \cite{RaSi}, $$\text{Tr}(d\alpha_k\delta e^{-t \Delta_{k, B}})= \text{Tr}(\alpha_k\delta d e^{-t \Delta_{k-1, B}})\ \ {\rm and} \ \  \text{Tr}(\delta \alpha_k d  e^{-t \Delta_{k, B}})=\text{Tr}( \alpha_k d \delta e^{-t \Delta_{k+1, B}}).$$ Thus, setting $\theta_k:=\text{Tr}(\alpha_k\delta d e^{-t \Delta_{k, B}})$ and $\varphi_k:=\text{Tr}( \alpha_k d \delta e^{-t \Delta_{k, B}})$, we obtain
\begin{align*}
\frac{\partial}{\partial u} f(u,s)  =\sum_{k=0}^n (-1)^{k+1} \beta_k\ \int_0^\infty t^{s} (\varphi_{k+1} - \theta_k + \varphi_k - \theta_{k-1})dt
\end{align*}
 and hence, by manipulating as in the closed manifold case and integrating by parts, 
\begin{align*}
\frac{\partial}{\partial u}f(u,s)& =  \sum_{k=1}^{n}(-1)^{k}(\beta_{k+1} - 2\beta_k + \beta_{k-1})s\zeta(\alpha_k d \delta\Delta_{k, B}^{-1},\Delta_{k, B}, s)\\
&+  \sum_{k=0}^{n}(-1)^{k+1}(\beta_{k+1}-\beta_k) s\zeta(\alpha_k,\Delta_{k, B}, s),
\end{align*}
with, as before, the $k=0$ and $k=n$ cases needing separate treatment but which both vanish. By (1.14) in \cite{GrSc}
\begin{align*}
\frac{\partial}{\partial u}f(u,0)&= \frac{1}{2} \sum_{k=1}^{n}(-1)^{k}(\beta_{k+1} - 2\beta_k + \beta_{k-1})\res(\alpha_k d \delta\Delta_{k, B}^{-1})\\
&+ \frac{1}{2}  \sum_{k=0}^{n}(-1)^{k+1}(\beta_{k+1}-\beta_k) \res(\alpha_k).
\end{align*}
Since $\alpha_k$ is the usual multiplication operator, $\res(\alpha_k)=0$ and the right hand side vanishes as for the closed manifold case if and only if  $\beta_{k+1} - 2\beta_k + \beta_{k-1}=0$ for each non-negative integer $k$. Proceeding similarly for the metric on $E_\rho$, we thus obtain that the torsions 
$T^{\res, \beta}_{X,\, \mathcal{A}}(\rho), T^{\res, \beta}_{X, \mathcal{R}}(\rho), T^{\zeta, \beta}_{X, \,\mathcal{A}}(\rho), T^{\zeta, \beta}_{X, \mathcal{R}}(\rho)$ are smooth invariants precisely in the cases $\beta= {\bf 1}$ or ${\bf k}$. 

For the exact formulae \eqref{antorsionunweighted} - \eqref{leading trace torsion weighted 2} we proceed as follows. Since $*\mathcal{R}=\mathcal{A}*$, we have \begin{align}\label{unfortunatelynottrueanumoreforboundary10}
\zeta_{k, \mathcal{R}}(s)=\zeta_{n-k, \mathcal{A}}(s)
\end{align}
and, as both $\Delta_{k, \mathcal{R}}$ and $\Delta_{k, \mathcal{A}}$  have a discrete set of non-negative eigenvalues accumulating at infinity,  we can adapt the argument of Theorem 2.3 in \cite{RaSi} to see that:
\begin{prop}
\begin{align}\label{unfortunatelynottrueanumoreforboundary}
\sum_{k=0}^n(-1)^k k\zeta_{k, \mathcal{R}}(s) = (-1)^{n-1} \sum_{k=0}^n(-1)^k k \zeta_{k, \mathcal{A}}(s).
\end{align}
\end{prop}

\begin{proof}
Let $\lambda \neq 0$ be an eigenvalue for $\Delta_{k, \mathcal{R}}$ with associated eigenspace  
\begin{align*}
\mathcal{E}_{k, \mathcal{R}}(\lambda)=\{\omega \in \Omega^k(X,E_\rho)|\ \Delta\omega=\lambda\omega, \mathcal{R}\gamma\omega=\mathcal{R}\gamma\delta\omega=0\}.
\end{align*}
Then
%\begin{align*}
$\Lambda_k'(\lambda)=\frac{1}{\lambda}d\delta$  and $\Lambda_k''(\lambda)=\frac{1}{\lambda}\delta d$
%\end{align*}
are respective orthogonal projections of $\mathcal{E}_{k, \mathcal{R}}(\lambda)$ onto $\mathcal{F}_{k, \mathcal{R}}(\lambda)=\{\omega \in \mathcal{E}_{k, \mathcal{R}}(\lambda)|\ d\omega=0\}$ and $\mathcal{G}_{k, \mathcal{R}}(\lambda)=\{\omega \in \mathcal{E}_{k, \mathcal{R}}(\lambda)|\ \delta\omega=0\}$. Also, by construction, ${\Lambda_k'(\lambda)+\Lambda_k''(\lambda)=I}$. Since the map  $\frac{1}{\sqrt{\lambda}}d$ is an isomorphism with  inverse $\frac{1}{\sqrt{\lambda}}\delta$, we conclude $\mathcal{G}_{k, \mathcal{R}}(\lambda) \cong \mathcal{F}_{k+1, \mathcal{R}}(\lambda)$ and thence
\begin{align*}
g_{k, \mathcal{R}}(\lambda)=|\mathcal{G}_{k, \mathcal{R}}(\lambda)|=| \mathcal{F}_{k+1, \mathcal{R}}(\lambda)|=f_{k+1, \mathcal{R}}(\lambda).
\end{align*} 
Therefore
\begin{align*}
\zeta_{k, \mathcal{R}}(s)&=\sum_{\lambda\neq 0} \lambda^{-s}|\mathcal{E}_{k, \mathcal{R}}(\lambda)|=\sum_{\lambda\neq 0} \lambda^{-s}(f_{k, \mathcal{R}}(\lambda) + f_{k+1, \mathcal{R}}(\lambda))\\
&=\sum_{\lambda\neq 0} \lambda^{-s}(g_{k, \mathcal{R}}(\lambda) + g_{k-1, \mathcal{R}}(\lambda)) \qquad 	\text{ and}\\
\sum_{k=0}^n(-1)^k k \zeta_{k, \mathcal{R}}(s)&=\sum_{k=1}^n(-1)^k  \sum_{\lambda\neq 0} \lambda^{-s}f_{k, \mathcal{R}}(\lambda)=-\sum_{k=0}^{n-1}(-1)^k  \sum_{\lambda\neq 0} \lambda^{-s}g_{k, \mathcal{R}}(\lambda).
\end{align*}
 $*\mathcal{R}=\mathcal{A}*$ yields  $\mathcal{F}_{k, \mathcal{R}}(\lambda)\cong \mathcal{G}_{n-k, \mathcal{A}}(\lambda)$ and therefore $f_{k, \mathcal{R}}(\lambda)=g_{n-k, \mathcal{A}}(\lambda)$, which proves  (\ref{unfortunatelynottrueanumoreforboundary10}). Thus, in conclusion
\begin{align*}
\sum_{k=0}^n(-1)^k k \zeta_{k, \mathcal{R}}(s)&=-\sum_{k=0}^{n-1}(-1)^k  \sum_{\lambda\neq 0} \lambda^{-s}g_{k, \mathcal{R}}(\lambda)=\sum_{k=1}^n(-1)^k  \sum_{\lambda\neq 0} \lambda^{-s}f_{k, \mathcal{R}}(\lambda)\\
&=\sum_{k=1}^n(-1)^k  \sum_{\lambda\neq 0} \lambda^{-s}g_{n-k, \mathcal{A}}(\lambda)=\sum_{k=0}^{n-1}(-1)^{n-k}  \sum_{\lambda\neq 0} \lambda^{-s}g_{k, \mathcal{A}}(\lambda)\\
&=(-1)^n\sum_{k=0}^{n-1}(-1)^k  \sum_{\lambda\neq 0} \lambda^{-s}g_{k, \mathcal{A}}(\lambda) =(-1)^{n-1} \sum_{k=0}^n(-1)^k k \zeta_{k, \mathcal{A}}(s).
\end{align*}
\end{proof} 

Combining  (\ref{unfortunatelynottrueanumoreforboundary10}) and (\ref{unfortunatelynottrueanumoreforboundary})  yields
\begin{align}\label{unweighedzetavanishes}
  \sum_{k=0}^n(-1)^k \zeta_{k, \mathcal{R}}(s)=(-1)^n  \sum_{k=0}^n(-1)^k \zeta_{k, \mathcal{A}}(s)= \frac{1}{n}\sum_{k=0}^n(-1)^k k\left( \zeta_{k, \mathcal{R}}(s) + (-1)^n\zeta_{k, \mathcal{A}}(s)\right) =0, 
\end{align}
which implies \eqref{antorsionunweighted}, while the  Euler  and derived Euler characteristics
\begin{align*}
\chi_\mathcal{B} (X,E_\rho):= \sum_{k=0}^n (-1)^k   \dim H^k_\mathcal{B}(X,E_\rho), \quad % \chi_\mathcal{B}(X,E_\rho):=\sum_{k=0}^n (-1)^k  \dim H^k_\mathcal{B}(X,E_\rho), 
%\\
\chi_\mathcal{B}'(X,E_\rho):= \sum_{k=0}^n (-1)^k k  \dim H^k_\mathcal{A}(X,E_\rho),% \quad \chi_\mathcal{R}'(X,E_\rho):=\sum_{k=0}^n (-1)^k k \dim H^k_\mathcal{R}(X,E_\rho).
\end{align*}
are identified using the Hodge Theorem on a compact manifold with boundary  $$\ker(\Delta_{k,\mathcal{R}}^{})\cong H^k_\mathcal{R}(X,E_\rho)\cong H^k(X,Y,E_\rho),\ \ \ \ \ \ker(\Delta_{k,\mathcal{A}}^{})\cong H^k_\mathcal{A}(X,E_\rho)\cong H^k(X,E_\rho),$$
 with 
\begin{align*}
\log T_{X,B}^{\,\res, {\bf 1}}(\rho) &  \stackrel{(\ref{ScottsformulabutnowGrubb})}{=}  \sum_{k=0}^{n} (-1)^k \zeta_{k, B}(0) + \sum_{k=0}^{n} (-1)^k \dim \ker(\Delta_{k,B}^{}) \stackrel{(\ref{unweighedzetavanishes})}{=} 0 + \chi_B(X,E_\rho),
\end{align*}
and 
\begin{align*}
\log T_{X,B}^{\,\res, {\bf k}}(\rho) & = \sum_{k=0}^{n} (-1)^k k\ \zeta_{k, B}(0) + \sum_{k=0}^{n} (-1)^k k \dim \ker(\Delta_{k,B}^{}) \\
&= \sum_{k=0}^{n} (-1)^k k\ \zeta_{k, B}(0)  + \chi_B'(X,E_\rho)
\end{align*}
which is \eqref{leading trace torsion weighted}.

Finally, as $\sum_{k=0}^{n} (-1)^k k \left(\zeta_{k, B}(0)  +\dim \ker(\Delta_{k,B}^{})\right)$ was shown to equal $\frac{1}{2} \dim (X) \,\chi_B(X, E_\rho)$ in  \cite{Vi} (Proposition 2.23), we have  \eqref{leading trace torsion weighted 2} and therefore (\ref{Tres2}) generalises to the boundary case for absolute/relative boundary conditions. 

Since the formulae on the right-hand side of \eqref{leading trace torsion} and \eqref{leading trace torsion weighted} are topological invariants of $(X,E_\rho)$ 
we may infer, as for the closed manifold case, that $T^{\res, \beta}_{X,\, \mathcal{A}}(\rho)$, $T^{\res, \beta}_{X, \mathcal{R}}(\rho)$,  are topological invariants of $(X,E_\rho)$  precisely in the cases $\beta= {\bf 1}$ or ${\bf k}$.  From \eqref{antorsionunweighted}, $T^{\zeta, {\bf 1}}_{X, \, B}(\rho)$ is trivially topological. Since $R_{X,\mathcal{A}}(\rho)$ is a topological invariant \cite{Chap} and $R_{X,\mathcal{R}}(\rho)$ is a PL-invariant \cite{Mi61}, then from \eqref{antorsionvishik} $T^{\zeta, {\bf k}}_{X, \mathcal{A}}(\rho)$ is topological while $T^{\zeta, {\bf k}}_{X, \mathcal{R}}(\rho)$ is a PL-invariant. Conversely, if  this holds then it holds smoothly and hence up to a constant multiple the only possibilities are $\beta = {\bf 1}$ or ${\bf k}$.

We remark that, unlike for even dimensional closed manifolds, $\sum_{k=0}^n(-1)^k  k\zeta_{k, \mathcal{R}}(s)$ need not vanish in general. For example, let $X=[0, R]$, where the eigenvalue problem for $\Delta_0=-\partial^2_x$ with relative boundary conditions is  the well-known harmonic oscillator with Dirichlet boundary conditions, whose  eigenvalues are $\lambda=\frac{n^2\pi^2}{R^2}$, $n \in \mathbb{N}$. Hence, for $\zeta_{\mathbf{R}}(s)$ the \emph{Riemann zeta function},
\begin{align*}
\zeta_{0, \mathcal{R}}(s)= 2\frac{R^{2s}}{\pi^{2s}}\sum_{n=1}^\infty n^{-2s}=2\frac{R^{2s}}{\pi^{2s}}\zeta_{\mathbf{R}}(2s).
\end{align*}
Consequently,
\begin{align*}
\sum_{k=0}^1 (-1)^k k \zeta_{k, \mathcal{A}}(s)= -  \zeta_{1, \mathcal{A}}(s) \stackrel{(\ref{unfortunatelynottrueanumoreforboundary10})}{=} -  \zeta_{0, \mathcal{R}}(s) = -2 \frac{R^{2s}}{\pi^{2s}}\zeta_{\mathbf{R}}(2s)
\end{align*}
is non-zero and $\sum_{k=0}^1 (-1)^k k \zeta_{k, \mathcal{A}}(0)= -2  \zeta_{\mathbf{R}}(0)  =1$. For an example in dimension 2, let $X$ be the cylinder $[0,R] \times S^1$, with  $x \in [0, R]$  the normal coordinate, and  $\Delta= -\partial^2_x + \Delta^{S^1}$. Since, $\zeta_{1, \mathcal{R}}(s)= \zeta_{0, \mathcal{R}}(s) + \zeta_{2, \mathcal{R}}(s)$ by (\ref{unweighedzetavanishes}), we obtain
$$\sum_{k=0}^2 (-1)^k k \zeta_{k, \mathcal{R}}(s)  = - \zeta_{1, \mathcal{R}}(s) + 2 \zeta_{2, \mathcal{R}}(s)  =\zeta_{2, \mathcal{R}}(s) - \zeta_{0, \mathcal{R}}(s)  \stackrel{(\ref{unfortunatelynottrueanumoreforboundary10})}{=}  \zeta_{0, \mathcal{A}}(s) - \zeta_{0, \mathcal{R}}(s)= \zeta_{0}^{S^1}(s),$$
where the last equality is shown in \cite{KiLoPa} (\S 3.2), as $\Delta_0$ with relative/absolute boundary conditions corresponds to the Laplacian on functions with Dirichlet/Neumann conditions. Hence, $\sum_{k=0}^2 (-1)^k k \zeta_{k, \mathcal{R}}(0) =\zeta_{0}^{S^1}(0)= 2\zeta_{\mathbf{R}}(0)=-1$.\\

From \cite{Vi}, Proposition 2.22, and (\ref{ScottsformulabutnowGrubb}), we have that the residue torsions satisfy the following gluing formula. Let $X=X_1 \cup_Y X_2$ with $Y \cap \partial X=\emptyset$, absolute/relative boundary conditions on $\partial X$, and relative boundary conditions on $Y$. Then
\begin{align}
    \log  T_{X,B}^{\,\res, {\bf k}}(\rho)= \log  T_{X_1,B}^{\,\res, {\bf k}}(\rho) + \log  T_{X_2,B}^{\,\res, {\bf k}}(\rho) + \log  T_{Y}^{\,\res, {\bf k}}(\rho) + \frac{1}{2}\chi(Y, {E_\rho}_{|Y}).
\end{align}
This completes the proof of \thmref{thm two}.

\vfive

{\small \noi \textsc{Department of Mathematics \\
King's College London }}}
\end{document}